\documentclass[12pt,a4paper,french,english]{article}
\usepackage[top=3cm, bottom=3cm, left=3cm, right=3cm]{geometry}
\usepackage[T1]{fontenc}
\usepackage{lmodern}
\usepackage{babel}
\usepackage[latin1]{inputenc}
\usepackage{csquotes}
\usepackage{amsmath}
\usepackage{amsthm}
\usepackage{amsfonts}
\usepackage{amssymb}
\usepackage{graphics}
\usepackage{float}
\usepackage[hang,flushmargin]{footmisc} 

\usepackage{amsmath,amsthm,amssymb,graphicx, multicol, array}
\usepackage{enumerate}
\usepackage{enumitem}
\usepackage{xcolor}
\usepackage{multicol}

\usepackage[pagebackref,bookmarks,colorlinks,breaklinks,linktoc=page]{hyperref}
\hypersetup{linkcolor=blue,citecolor=red,filecolor=blue,urlcolor=blue} 

\usepackage{tabto}

\numberwithin{equation}{section}

\usepackage{tikz}
\usepackage{pgfplots}

\usetikzlibrary{matrix}

\usepackage{mathrsfs}

\DeclareMathOperator{\li}{li}
\DeclareMathOperator{\ord}{ord}

\newtheorem{thm}{Theorem}[section]
\newtheorem{lem}{Lemma}[section]

\newtheorem{dfn}{Definition}[section]
\newtheorem{exe}{Exercise}[section]

\newcommand{\N}{\mathbb{N}}
\newcommand{\Z}{\mathbb{Z}}

\newcommand{\C}{\mathbb{C}}
\newcommand{\F}{\mathbb{F}}

\usepackage{fancyhdr}
\usepackage{titletoc}
\let\LaTeXStandardTableOfContents\tableofcontents

\renewcommand{\tableofcontents}{%
	\begingroup%
	\renewcommand{\bfseries}{\relax}%
	\LaTeXStandardTableOfContents%
	\endgroup%
}%

\title{Primitive Roots In Short Intervals}
\date{}
\author{N. A. Carella}

\begin{document}
\thispagestyle{empty}
\date{}
\maketitle

\vskip .25 in 
\textbf{\textit{Abstract}:} Let \(p \geq 2\) be a large prime, and let $N \gg (\log p)^{1+\varepsilon} $. This note proves the existence of primitive roots in the short interval $[M,M+N]$, where $M \geq 2$ is a fixed number, and $ \varepsilon>0$ is a small number. In particular, the least primitive root $g(p)= O\left (  (\log p)^{1+\varepsilon} \right )$, and the least prime primitive root $g^*(p)= O\left (  ( \log p)^{1+\varepsilon} \right )$ unconditionally.  \let\thefootnote\relax\footnote{\today \date{} 
}

\vskip .25 in 
\tableofcontents

\section{Introduction}
Given a large prime $p \geq 2$, and a number $N \leq p$. The standard analytic methods demonstrate the existence of primitive roots in any short interval 
\begin{equation} \label{eq175.03}
\left [M, M+N \right ]
\end{equation}
for any number $N \gg p^{1/2+\varepsilon} $, where $M \geq 2$ is a fixed number, and $ \varepsilon>0$ is a small number, see \cite{ES57}, \cite{DH37}, \cite{CL53}, \cite{PS90}. More elaborate exponential sums methods can reduce the size of the interval to $N \gg p^{1/4+\varepsilon}$, see \cite{BD67}. Further, the explicit upper bound claims that the least primitive root $g(p) \geq 2$ satisfies the inequality 
\begin{equation} \label{eq175.13}
g(p) <\sqrt{p}-2 
\end{equation}
for all primes $p >409$, see \cite{CT15}, and \cite{MT15}. Assuming the GRH, it was proved that $g(p) =O\left ( \log ^6 p\right )$, and the average value is $\overline{g(p)} =O\left ( (\log \log p)^2 \right )$, see \cite{SV92} and \cite{BE93} respectively. \\

Almost all these results are based on the standard indicator function in \hyperlink{lem333.2}{Lemma} \ref{lem333.2}. This note introduces a new technique based on the indicator function in \hyperlink{lem333.3}{Lemma} \ref{lem333.3} to improve the results for primitive roots in short intervals.

\begin{thm} \label{thm1.1}\hypertarget{thm1.1} Given a small number $ \varepsilon>0$, and a sufficiently large prime \(p \geq 2\), let $N \gg (\log p)^{1+\varepsilon}$. Then, the short interval 
\begin{equation} \label{el03}
\left [ M, M+  N\right ]
\end{equation}
contains a primitive root for any fixed $M \geq 2$. In particular, the least primitive root $g(p) =O\left (  (\log p)^{1+\varepsilon} \right )$ unconditionally. 
\end{thm} 

As the probability of a primitive root modulo $p$ is $O(1/\log \log p)$, this result is nearly optimal, see Section \ref{s222} for a discussion.\\

The existence of prime primitive roots in short interval $[M,M+N]$ requires information about primes in short intervals such that $N < p^{1/2}$, and $M \geq 2$ is any fixed number, which is not available in the literature. But, for the long interval $[2, x]$, it is feasible. Recently, it was proved that the least prime primitive root $ g^*(p)= O\left (  p^{\varepsilon} \right )$, unconditionally, see \cite{CN17}. Moreover, assuming standard conjectures, the least prime primitive root is expected to be $g^{*}(p) =O\left ( (\log p) (\log \log p)^2 \right )$, see \cite{BE97}. A very close upper bound is provided here.

\begin{thm} \label{thm1.2}\hypertarget{thm1.2} If \(p \geq 2\) is a sufficiently large prime, then, the least prime primitive root satisfies
\begin{equation} \label{el05}
g^{*}(p) =O\left (  (\log p)^{1+\varepsilon} \right )
\end{equation}
for any small number $ \varepsilon>0$, unconditionally. 
\end{thm} 

\begin{thm} \label{thm1.3}\hypertarget{thm1.3} Let \(p \geq 2\) be a sufficiently large prime, and let $N \gg p^{.535}$. Then, the short interval 
\begin{equation} \label{el07}
\left [ M, M+  N\right ]
\end{equation}
contains a prime primitive root for any fixed $M \geq 2$ unconditionally. 
\end{thm}

The fundamental background materials are discussed in the earlier sections. \hyperlink{S887}{Section} \ref{S887} presents a proof of \hyperlink{thm1.1}{Theorem} \ref{thm1.1}, the penultimate section presents a proofs of \hyperlink{thm1.2}{Theorem} \ref{thm1.2}, and the last section presents a proof of \hyperlink{thm1.3}{Theorem} \ref{thm1.3}. \\


\section{Primitive Roots Test} \label{S969}\hypertarget{S969}
For a prime $p \geq 2$, the multiplicative group of the finite fields $\mathbb{F}_p$ is a cyclic group for all primes. 

\begin{dfn} \label{dfn969.29}\hypertarget{dfn969.29}{ \normalfont The order $\min \{k \in \mathbb{N}: u^k \equiv 1 \bmod p \}$ of an element $u \in \mathbb{F}_p$ is denoted by $\ord_p(u)$. An element is a \textit{primitive root} if and only if $\ord_p(u)=p-1$. }
\end{dfn}
The Euler totient function counts the number of relatively prime integers \(\varphi (n)=\#\{ k:\gcd (k,n)=1 \}\). This counting function is compactly expressed by the analytic formula \(\varphi (n)=n\prod_{p \mid n}(1-1/p),n\in \mathbb{N} .\)
\begin{lem} {\normalfont (Fermat-Euler)} \label{lem2.1}If \(a\in \mathbb{Z}\) is an integer such that \(\gcd (a,n)=1,\) then \(a^{\varphi (n)}\equiv
	1 \bmod n\).
\end{lem}
\begin{lem} \label{lem969.05}\hypertarget{lem969.05}  {\normalfont (Primitive root test)} An integer $u \in \Z$ is a primitive root modulo an integer $n \in \N$ if and only if 
\begin{equation*}\label{eq969.52}
u^{\varphi (n)/p} -1\not \equiv 0 \mod  n
\end{equation*}
for all prime divisors $p \mid \varphi (n)$.
\end{lem}
The primitive root test is a special case of the Lucas primality test, introduced in {\color{red}\cite[p.\ 302]{ LE78}}. A more recent version appears in {\color{red}\cite[Theorem 4.1.1]{CP05}}, and similar sources. 
\begin{lem} \label{lem969.21}\hypertarget{lem969.21}  {\normalfont (Complexity of primitive root test)} Given a prime $p \geq 2$, and the squarefree part $p_1 p_2 \cdots p_v \mid p-1$, a primitive root modulo $p$ can be determined in deterministic polynomial time $O(\log ^c p)$, some constant $c >1$.
\end{lem}
\begin{proof}[\textbf{Proof}] The mechanics of the deterministic polynomial time algorithm are specified in {\color{red}\cite[Chapter 11]{SV08}}. By \hyperlink{thm1.2}{Theorem} \ref{thm1.2}, the algorithm is repeated at most $O\left (  (\log p)^{1+\varepsilon} \right )$ times for each $u=O\left (  (\log p)^{1+\varepsilon} \right )$. These prove the claim.
\end{proof}

\section{Representations of the Characteristic Functions} \label{s333}
The characteristic function \(\Psi :G\longrightarrow \{ 0, 1 \}\) of primitive elements is one of the standard analytic tools employed to investigate the various properties of primitive roots in cyclic groups \(G\). Many equivalent representations of the characteristic function $\Psi $ of primitive elements are possible. Several of these representations are studied in this section.

\subsection{Divisors Dependent Characteristic Function}
A representation of the characteristic function dependent on the orders of the cyclic groups is given below. This representation is sensitive to the primes decompositions $q=p_1^{e_1}p_2^{e_2}\cdots p_t^{e_t}$, with $p_i$ prime and $e_i\geq1$, of the orders of the cyclic groups $q=\# G$. \\

\begin{lem} \label{lem333.2}\hypertarget{lem333.2}
Let \(G\) be a finite cyclic group of order \(p-1=\# G\), and let \(0\neq u\in G\) be an invertible element of the group. Then
\begin{equation} \label{eq333.02}
\Psi (u)=\frac{\varphi (p-1)}{p-1}\sum _{d \mid p-1} \frac{\mu (d)}{\varphi (d)}\sum _{\ord(\chi ) = d} \chi (u)=
\left \{\begin{array}{ll}
1 & \text{ if } \ord_p (u)=p-1,  \\[.3cm]
0 & \text{ if } \ord_p (u)\neq p-1. \\
\end{array} \right .\nonumber
\end{equation}
\end{lem}

\begin{proof}[\textbf{Proof}] Assume that $u=\tau^{qm}$ is a $q$th power residue modulo $p$, where $q\mid p-1$ and $\gcd(m,p-1)=1$. Then, the inner sum
\begin{equation} 
\sum _{ \ord(\chi) = q} \chi (u)= \sum _{ \ord(\chi) = q} \chi (\tau^{qm})=\sum _{ \ord(\chi) = q} \chi (\tau^{m})^q=\varphi(q)=q-1,
\end{equation}
 where $\chi(v)^q=1$. Replacing this information into the product
\begin{eqnarray} 
\frac{\phi(p-1)}{p-1} \sum_{d \mid p-1}\frac{\mu(d)}{\varphi(d)} \sum_{\ord(\chi)=d}\chi(u)
&=&\frac{\phi(p-1)}{p-1} \prod_{q \mid p-1} \left (1- \frac{\sum_{\ord(\chi)=q}\chi(u)}{q-1} \right )  \nonumber \\[.3cm]
&=&\frac{\phi(p-1)}{p-1} \prod_{q \mid p-1} \left (1- \frac{q-1}{q-1} \right )=0 .
\end{eqnarray}
shows that both sides of the equation vanish if the element $u \in G$ has order $\ord_p(u) =q \mid p-1$ and $q < p-1$. Now, assume that $u=\tau^{m}$ is not $q$th power residue modulo $p$ for any $q \mid p-1$, where $\gcd(m,p-1)=1$. Then, the inner sum
\begin{equation} 
\sum _{ \ord(\psi) = q} \chi (u)= \sum _{ \ord(\psi) = q} \chi (\tau^{m})=-1.
\end{equation}
Replacing this information into the product
\begin{eqnarray} 
\frac{\phi(p-1)}{p-1} \sum_{d \mid p-1}\frac{\mu(d)}{\varphi(d)} \sum_{\ord(\chi)=d}\chi(u)
&=&\frac{\phi(p-1)}{p-1} \prod_{q \mid p-1} \left (1- \frac{\sum_{\ord(\chi)=q}\chi(u)}{q-1} \right )  \nonumber \\[.3cm]
&=&\frac{\phi(p-1)}{p-1} \prod_{q \mid p-1} \left (1- \frac{-1}{q-1} \right )=1 .
\end{eqnarray}
These verify that both sides of the equation vanishes if and only if the element $u \in G$ has order $\ord_p(u) =q \mid p-1$ and $q < p-1$.
\end{proof}
	
The precise source of formula \eqref{eq333.02} is not clear. The authors in \cite{DH37}, and \cite{WR01} attributed this formula to Vinogradov, and other authors have attributed it to Landau, \cite{LE1927}. The proof and other details on the characteristic function are given in {\color{red}\cite[p. 863]{ES57}}, {\color{red}\cite[p.\ 258]{LN97}}, {\color{red}\cite[p.\ 18]{MP07}}. The characteristic function for multiple primitive roots is used in {\color{red}\cite[p.\ 146]{CZ98}} to study consecutive primitive roots. In \cite{DS12} it is used to study the gap between primitive roots with respect to the Hamming metric. And in \cite{WR01} it is used to prove the existence of primitive roots in certain small subsets \(A\subset \mathbb{F}_p\). In \cite{DH37} it is used to prove that some finite fields do not have primitive roots of the form $a\tau+b$, with $\tau$ primitive and $a,b \in \mathbb{F}_p$ constants. In addition, the Artin primitive root conjecture for polynomials over finite fields was proved in \cite{PS95} using this formula.

\subsection{Divisors Free Characteristic Function}
It often difficult to derive any meaningful result using the usual divisors dependent characteristic function of primitive elements given in \hyperlink{lem333.2}{Lemma} \ref{lem333.2}. This difficulty is due to the large number of terms that can be generated by the divisors, for example, \(d\mid p-1\), involved in the calculations, see \cite{ES57}, \cite{DS12} for typical applications and {\color{red}\cite[p.\ 19]{MP04}} for a discussion. \\
	
A new \textit{divisors-free} representation of the characteristic function of primitive element is developed here. This representation can overcomes some of the limitations of its counterpart in certain applications. The \textit{divisors dependent representation} of the characteristic function of primitive roots, \hyperlink{lem333.2}{Lemma} \ref{lem333.2}, detects the order \(\ord_p (u)\) of the element \(u\in \mathbb{F}_p\) by means of the divisors of the totient \(p-1\). In contrast, the \textit{divisors-free representation} of the characteristic function, \hyperlink{lem333.3}{Lemma} \ref{lem333.3}, detects the order \(\text{ord}_p(u) \geq 1\) of the element \(u\in \mathbb{F}_p\) by means of the solutions of the equation \(\tau ^n-u=0\) in \(\mathbb{F}_p\), where \(u,\tau\) are constants, and \(1\leq n<p-1, \gcd (n,p-1)=1,\) is a variable. 
\begin{lem} \label{lem333.3}\hypertarget{lem333.3}
Let \(p\geq 2\) be a prime, and let \(\tau\) be a primitive root mod \(p\). If \(u\in\mathbb{F}_p\) is a nonzero element, and \(\psi \neq 1\) is a nonprincipal additive character of order \(\ord \psi =p\), then
\begin{equation}
\Psi (u)=\sum _{\gcd (n,p-1)=1} \frac{1}{p}\sum _{0\leq m\leq p-1} \psi \left ((\tau ^n-u)m\right)=\left \{
\begin{array}{ll}
1 & \text{ if } \ord_p(u)=p-1,  \\[.3cm]
0 & \text{ if } \ord_p(u)\neq p-1. \\
\end{array} \right .\nonumber
\end{equation}
\end{lem}

\begin{proof}[\textbf{Proof}] As the index \(n\geq 1\) ranges over the integers relatively prime to \(p-1\), the element \(\tau ^n\in \mathbb{F}_p\) ranges over the primitive roots \(\text{mod } p\). Ergo, the equation
\begin{equation}\label{eq33.30}
\tau ^n- u=0
\end{equation} 
has a solution if and only if the fixed element \(u\in \mathbb{F}_p\) is a primitive root. Next, replace \(\psi (z)=e^{i 2\pi  z/p }\) to obtain
\begin{equation}
\Psi(u)=\sum_{\gcd (n,p-1)=1} \frac{1}{p}\sum_{0\leq m\leq p-1} e^{i 2\pi  (\tau ^n-u)m/p }=\left \{
\begin{array}{ll}
1 & \text{ if } \ord_p (u)=p-1,  \\[.3cm]
0 & \text{ if } \ord_p (u)\neq p-1. \\
\end{array} \right.
\end{equation}
This follows from the geometric series identity $\sum_{0\leq m\leq N-1} w^{ m }=(w^N-1)/(w-1)$ with $w \ne 1$, applied to the inner sum. 
\end{proof}

\section{Primes Numbers Results} \label{S533}\hypertarget{S533}
Some prime numbers results focusing on the local minima of the ratio
\begin{equation}\label{eq533.30}
\frac{\varphi(n)}{n}=\prod_{p \mid n}\left( 1- \frac{1}{p} \right)> \frac{1}{e^{\gamma} \log \log n+5/(2 \log \log n)}
\end{equation} 

are recorded in this section. The conditional results are studied in \cite{NJ12}, and the 
unconditional results are proved by various authors as {\color{red}\cite[Theorem 7 and Theorem 15]{RS62}}, and {\color{red}\cite[Theorem 2.9]{MV07}}.

\begin{lem} \label{lem533.01}\hypertarget{lem533.01}
Let \(n\geq 1\) be a large integer, and let $\omega(n)$ be the number of prime divisors $p \mid n$. Then
\begin{enumerate}[font=\normalfont, label=(\roman*)]
\item $\displaystyle\omega(n) \ll \log \log n,$ \tabto{6cm}the average number of prime divisors.
    
 \item $\displaystyle \omega(n) \ll \log n/ \log \log n,$\tabto{6cm}the maximal number of prime divisors.
 \end{enumerate}

\end{lem}
\begin{proof}[\textbf{Proof}] These are standard results in analytic number theory, see {\color{red}\cite[Theorem 2.6]{MV07}}.
\end{proof}
\begin{lem} \label{lem533.21}\hypertarget{lem533.21}
Let \(x\geq 2\) be a large number, then
\begin{enumerate}[font=\normalfont, label=(\roman*)]
\item $\displaystyle \prod_{p \leq x}\left( 1- \frac{1}{p} \right) 
=\frac{1}{e^{\gamma} \log x}+ O\left (e^{-c_0 \sqrt{ \log x}}\right ),$ \tabto{8.5cm} unconditionally.
    
 \item $\displaystyle \prod_{p \leq x}\left( 1- \frac{1}{p} \right) 
=\frac{1}{e^{\gamma} \log x}+\Omega_{\pm} \left (\frac{\log \log \log x}{x^{1/2}} \right ),$\tabto{8.5cm}unconditional oscillation.

 \item $\displaystyle \prod_{p \leq x}\left( 1- \frac{1}{p} \right) 
=\frac{1}{e^{\gamma} \log x}+ O\left (\frac{\log x}{ x^{1/2}} \right ),$\tabto{8.5cm}conditional on the RH.
 \end{enumerate}
The symbol $\gamma$ is the Euler constant, and $c_0>0$ is an absolute constant. 
\end{lem}
The explicit estimates are given in {\color{red}\cite[Theorem 7]{RS62}}, and the results for products over arithmetic progression are proved in \cite{LZ07}, et alii. The nonquantitative unconditional oscillations of the error of the product of primes is implied by the work of Phragmen, {\color{red}\cite[p.\ 182]{NW00}}. Since then, various authors have developed quantitative versions, see \cite{RS62}, \cite{DP09}, et alii.


\section{Basic Statistics for Primitive Roots} \label{s222}\hyperlink{S222}
Some elementary information is provided in this section.
\subsection{Probability Of Primitive Roots}
The probability of primitive roots in a finite field $\F_p$ has the closed form $\varphi(p-1)/(p-1) \leq 1/2$. The maximal probability $\varphi(p-1)/(p-1) = 1/2$ occurs on the subset of Fermat primes 
\begin{equation}
\mathcal{F}=\{p=2^{2^n}+1: n \geq 0\}=\{3,5,17,257, 65537, \ldots \}. 
\end{equation}
This is followed by the subset of Germain primes
\begin{equation}
\mathcal{S}=\{p=2^aq+1: q \geq  2 \text{ is prime, and } a \geq 1 \}=\{5,7,11, 13, 23, 29, \ldots \}, 
\end{equation}
which has $\varphi(p-1)/p =(1/2)(1-1/q)$, et cetera. Some basic questions such as the sizes of these subsets of primes are open problems. In contrast, the minimal probabilities occur on the various subsets of primes with highly composite totients $p-1$. For example, the subset 
\begin{equation}
\mathcal{R}=\{p \geq 2: p-1=2^{v_2}\cdot 3^{v_3}\cdot 5^{v_5}\cdots q^{v_q}, \text{ and } v_i \geq 1\}=\{3,7,31, 191, \ldots \}. 
\end{equation}
In these cases, the probability function can have a complicated expression such as
\begin{equation} \label{eq222.8}
\frac{\varphi(p-1)}{p-1}\asymp\prod_{q \ll \log p}\left( 1- \frac{1}{q} \right) 
=\frac{1}{e^{\gamma} \log \log p}+\Omega_{\pm} \left (\frac{\log \log \log \log p}{(\log p)^{1/2}} \right ).
\end{equation}
This is derived from the standard results in \hyperlink{lem533.01}{Lemma} \ref{lem533.01}, and in \hyperlink{lem533.21}{Lemma} \ref{lem533.21}. Further, the average probability over all the primes $p \leq x$ is a well known constant
\begin{equation} \label{eq222.21}
a_0=\frac{1}{\pi(x)} \sum_{p \leq x}\frac{\varphi(p-1)}{p-1}=\prod_{p >2}\left( 1- \frac{1}{p(p-1)} \right) +o(1)= 0.3739558136 \ldots.
\end{equation}
The analysis of the average appears in \cite{SP69}, and an early numerical calculations is given in \cite{WJ61}. The distribution of primitive root for highly composite totients $p-1$ is approximately a Poisson distribution with parameter $\lambda>0$. For $k \geq 0$, and $1 \leq t \leq \delta \log \log p$, with $\delta >0$, the probability function has the asymptotic formula
\begin{equation} \label{eq222.3c3}
P_k(t) \sim e^{-\lambda} \frac{\lambda^k}{k!},
\end{equation}
confer {\color{red}\cite[Theorem 2]{CZ98}} for the finer details.

\subsection{Average Gap Between Primitive Roots}	
Let $p\geq2 $ be a prime, and let $g_1, g_2, \ldots, g_t$ be the sequence of primitive roots in increasing order, with $t=\varphi(p-1)$. Given a fixed prime $p\geq 2$, the average gap between a pair of consecutive primitive roots is defined by 
\begin{equation} \label{eq222.33f}
d_n=g_{n+1}-g_n=\frac{p-1}{\varphi(p-1)} \ll \log \log p.
\end{equation}
 \begin{lem} \label{lem222.41}\hypertarget{lem222.44} Let \(x\geq 1\) be a large number, then the average gap between consecutive primitive roots over all the primes $p \leq x$ is bounded by a constant. In particular, for any constant $c>2$,
\begin{equation} \label{eq222.33j}
\overline{d_n}=\prod_{p \geq 2}\left( 1- \frac{1}{(p-1)^2} \right) \li(x)+  O\left (\frac{x}{\log^{c-1} x}\right ).\nonumber
\end{equation}
\end{lem}

\begin{proof}[\textbf{Proof}] The identity $n/\varphi(n)=\sum_{d\mid n}\mu^2(d)/\varphi(d)$ is used here to compute the average over all the primes $p \leq x$:
\begin{eqnarray} \label{eq222.74}
\sum_{p \leq x}\frac{p-1}{\varphi(p-1)} &=&\sum_{p \leq x} \sum_{d\mid p-1}\frac{\mu^2(d)}{\varphi(d)} \\
&=& \sum_{d\leq x}\frac{\mu^2(d)}{\varphi(d)}\sum_{\substack{p \leq x\\ p \equiv 1 \bmod d}}  1\nonumber.
\end{eqnarray}
To apply the prime number theorem to the inner sum, use a dyadic partition

\begin{equation} \label{eq222.76}
 \sum_{d\leq x}\frac{\mu^2(d)}{\varphi(d)}\sum_{\substack{p \leq x\\ p \equiv 1 \bmod d}}  1= \sum_{d\leq \log^c x}\frac{\mu^2(d)}{\varphi(d)}\sum_{\substack{p \leq x\\ p \equiv 1 \bmod d}}  1+ \sum_{d\geq \log^cx}\frac{\mu^2(d)}{\varphi(d)}\sum_{\substack{p \leq x\\ p \equiv 1 \bmod d}}  1,
\end{equation}
where $c>0$ is an arbitrary constant. The first sum has the asymptotic expression
\begin{eqnarray} \label{eq222.78}
 \sum_{d\leq \log^C x}\frac{\mu^2(d)}{\varphi(d)}\sum_{\substack{p \leq x\\ p \equiv 1 \bmod d}}  1 
&=& \sum_{d\leq \log^C x}\frac{\mu^2(d)}{\varphi(d)}        \left(  \frac{\li(x)}{\varphi(d)}+ O\left (\frac{x}{\log^b x}\right )   \right )  \\
&=&\li(x) \sum_{d\geq 2}\frac{\mu^2(d)}{\varphi(d)^2}+ O\left (\frac{x}{\log^b x}\right )  \nonumber,
\end{eqnarray}
where $b>c+1$. The second sum has the asymptotic expression
\begin{equation} \label{eq222.80}
  \sum_{d\geq \log^Cx}\frac{\mu^2(d)}{\varphi(d)}\sum_{\substack{p \leq x\\ p \equiv 1 \bmod d}}  1\ll \frac{x}{\log^c x} \sum_{d\geq \log^cx}\frac{1}{\varphi(d)}= O\left (\frac{x}{\log^{c-1} x}\right ) ,
\end{equation}
Combining the last two expressions (\ref{eq222.78}) and (\ref{eq222.80}) completes the proof.

\end{proof}
The average gap between consecutive primitive roots is precise the value of the constant
\begin{equation} \label{eq222.82}
\prod_{p \geq 2}\left( 1- \frac{1}{(p-1)^2} \right) =2.82638409425598556075406 \ldots .
\end{equation}
\begin{lem} \label{lem222.47}\hypertarget{lem222.47} Let \(p\geq 1\) be a large prime, and let $t\leq\varphi(p-1)$ be a large number. Then, the  $g_1, g_2, \ldots, g_t$ be the sequence of primitive roots in increasing order are uniformly distributed over the interval $[2, p-2]$.
\end{lem}
\begin{proof}[\textbf{Proof}] Let $\tau\ne\pm1, v^2$ be a primitive root mod $p$ and consider the sequence of rational approximations 
	\begin{equation}
x_n=\frac{g_n}{p}=\left \{\frac{\tau^n}{p}\right \}\in (0,1),
	\end{equation}
where $n\geq1$, $\gcd(n,p-1)=1$ and $p^{1/2}< x\leq \varphi(p-1)$. Now apply the Bohl-Weyl criterion, followed by \hyperlink{thm3.4}{Theorem} \ref{thm3.4} to obtain
\begin{equation}
\frac{1}{x} \sum_{1\leq n \leq x} e^{i2\pi x_n} =\frac{1}{x} \sum_{1\leq n \leq x} e^{\frac{i2\pi  g_n}{p}} =\frac{1}{x} \sum_{\substack{1\leq n \leq x\\\gcd(n,p-1)=1}} e^{\frac{i2\pi  \tau^n}{p}} =o(1).
\end{equation} 
This proves the uniform distribution, see \cite{BP1909}, \cite{SW1910}, \cite{WH1916} for the earliest references.
\end{proof}

\section{Estimates of Exponential Sums} \label{S4400}\hypertarget{S4400}
This section provides simple estimates for the exponential sums of interest in this analysis. There are two objectives: To  determine an upper bound, proved in Theorem \ref{thm3.4}, and to show that
\begin{equation} \label{eq3.201}
\sum_{\gcd(n,p-1)=1} e^{i2\pi b \tau^n/p} =\sum_{\gcd(n,p-1)=1} e^{i2\pi \tau^n/p}+E(p),
\end{equation}
where $E(p)$ is an error term, this is proved in \hyperlink{lem333.22}{Lemma} \ref{lem333.22}. The proofs of these results are entirely based on established results and elementary techniques. 

\subsection{Incomplete And Complete Exponential Quadratic Sums}
The discrete logarithm 	$\log_{\tau}:\F^{\times}\longrightarrow \F^{\times}$ in a finite field $\F_p$ with respect to a  primitive root $\tau\ne\pm1, v^2$ mod $p$ is defined by
\begin{align}
			u\quad&\longrightarrow\quad \log_{\tau} u.
\end{align}
\begin{lem}   {\normalfont (Gauss sum)} \label{lem333.255}\hypertarget{lem333.255}  Let $p$ be a primes and let $\tau\ne\pm1, v^2$ be a primitive root mod $p$. Let $\chi(t)=e^{i2 \pi t/p} $ and  $\psi(t)=e^{i2\pi  \log_{\tau}t/(p-1)}$ be a pair of additive and multiplicative characters, respectively. Then, the Gaussian sum has the exact upper bound
	\begin{equation} \label{eq333-255}
		\left | \sum_{1 \leq s\leq p-1} \chi(t)\psi(t)\right  | = p^{1/2}.\nonumber
	\end{equation}
\end{lem} 
\begin{proof}[\textbf{Proof}] The proof of this classical result is widely available in the literature.
\end{proof}
\begin{lem}   {\normalfont (Mordell sum)} \label{lem333.27}\hypertarget{lem333.27}  Let $p$ be a large prime and let $\omega=e^{i2 \pi /p} $ and  $f(t)=e^{i2\pi  \tau^t/p}$ respectively. Then, Mordell sum has the upper bound
	\begin{equation} \label{eq3-355}
		\left | \sum_{1 \leq s\leq p-1} \omega^{-ts}e^{\frac{i2\pi b \tau^s}{p}} \right | \leq 3 q^{1/2} \log q.\nonumber
	\end{equation}
\end{lem} 
\begin{proof}[\textbf{Proof}] Set $\omega^{-ts}=e^{\frac{-i2\pi st}{p}}$. Now write
	\begin{eqnarray}\label{eqm333.27k} 
		\sum_{1 \leq s\leq p-1} \omega^{-ts}e^{\frac{i2\pi b \tau^s}{p}}&=&\sum_{1 \leq s\leq p-1} e^{\frac{i2\pi (-st+  b\tau^s)}{p}}  \nonumber \\[.3cm]
		&\leq & 2 p^{1/2} \log p+2p^{1/2}+1 \nonumber \\[.3cm]
		&\leq &3 p^{1/2} \log p, 
	\end{eqnarray}
	this follows from the upper bound of the Mordell sum in the second line of \eqref{eqm333.27k}, see \cite{CC2009}.
\end{proof}
This result states that the finite Fourier transform of the function $f(n)=e^{\frac{i2\pi (n+  \tau^n)}{p}}$ is quite similar to the Gauss sum even though $f(nm)\ne f(m)f(n) $ is not a multiplicative function. A proof for the version with two moduli $p$ and $q$ is given in {\color{red}\cite[Lemma 2]{BS2006}}. 

\subsection{Incomplete and Complete Exponential Sums}
Let $f: \C \longrightarrow \C$ be a function, and let $q \in \N$ be a large integer. The finite Fourier transform 
\begin{equation} \label{eq3.370}
\hat{f}(t)=\frac{1}{q} \sum_{0 \leq s\leq q-1} e^{i \pi st/q}
\end{equation}
and its inverse are used here to derive a summation kernel function, which is almost identical to the Dirichlet kernel.

\begin{dfn} \label{dfn3.23}\hypertarget{dfn3.23} {\normalfont Let $ p$ and $ q $ be primes, and let $\omega=e^{i 2 \pi/q}$, and $\zeta=e^{i 2 \pi/p}$ be roots of unity. The \textit{finite summation kernel} is defined by the finite Fourier transform identity
\begin{equation} \label{eq3.373}
\mathcal{K}(f(n))=\frac{1}{q} \sum_{0 \leq t\leq q-1,}  \sum_{0 \leq s\leq p-1} \omega^{t(n-s)}f(s)=f(n).\nonumber
\end{equation}
} 
\end{dfn}
This simple identity is very effective in computing upper bounds of some exponential sums
\begin{equation}
 \sum_{ n \leq x}  f(n)= \sum_{ n \leq x}  \mathcal{K}(f(n)),
\end{equation}
where $x  \leq p < q$. Two applications are illustrated here.

\begin{thm}  \label{thm3.2}\hypertarget{thm3.2} {\normalfont (\cite{SR73},  \cite{ML72}) }  Let \(p\geq 2\) be a large prime, and let \(\tau \in \mathbb{F}_p\) be an element of large multiplicative order $\ord_p(\tau) \mid p-1$. Then, for any $b \in [1, p-1]$,  and $x\leq p-1$,
\begin{equation}
 \sum_{ n \leq x}  e^{\frac{i2\pi b \tau^{n}}{p}} \ll p^{1/2}  \log p.\nonumber
\end{equation}

\end{thm}

\begin{proof}[\textbf{Proof}] Let $q=p$ be a large prime, and let $f(n)=e^{i 2 \pi b\tau^{n} /p}$, where $\tau$ is a primitive root modulo $p$. Applying the finite summation kernel in \hyperlink{dfn3.23}{Definition} \ref{dfn3.23}, yields
\begin{equation} \label{eq3.372}
\sum_{ n \leq x}  e^{i2\pi b \tau^{n}/p}= \sum_{ n \leq x}\frac{1}{p} \sum_{0 \leq t\leq p-1,}  \sum_{1 \leq s\leq p-1} \omega^{t(n-s)}e^{i2\pi b \tau^{s}/p} .
\end{equation}
The term $t=0$ contributes $-x/p$, and rearranging it yield
\begin{eqnarray} \label{eq3.374}
\sum_{ n \leq x}  e^{\frac{i2\pi b \tau^{n}}{p}}
&=&\frac{1}{q} \sum_{ n \leq x,} \sum_{1 \leq t\leq p-1,}  \sum_{1 \leq s\leq p-1} \omega^{t(n-s)}e^{\frac{i2\pi b \tau^{s}}{p}}-\frac{x}{q} \\[.3cm]
&=&\frac{1}{p}  \sum_{1 \leq t\leq p-1}  \left (\sum_{1 \leq s\leq p-1} \omega^{-ts}e^{\frac{i2\pi b \tau^{s}}{p}} \right ) \left (\sum_{ n \leq x}\omega^{tn} \right )-\frac{x}{p}\nonumber.
\end{eqnarray}
Taking absolute value, and applying \hyperlink{lem333.20}{Lemma} \ref{lem333.20}, and \hyperlink{lem333.27}{Lemma} \ref{lem333.27}, yield
\begin{eqnarray} \label{eq3.376}
\left | \sum_{ n \leq x}  e^{\frac{i2\pi b \tau^{n}}{p}} \right |
&\leq&\frac{1}{p}  \sum_{1 \leq t\leq p-1} \left | \sum_{0 \leq s\leq p-1} \omega^{-ts}e^{\frac{i2\pi b \tau^{s}}{p}} \right | \cdot  \left | \sum_{ n \leq x}\omega^{tn} \right |+ \frac{x}{p}\nonumber \\[.3cm]
&\ll&\frac{1}{p}  \sum_{1 \leq t\leq p-1} \left ( 2p^{1/2} \log p \right ) \cdot  \left ( \frac{2p}{\pi t} \right )+\frac{x}{p}\\[.3cm]
&\ll& p^{1/2} \log^2 p\nonumber .
\end{eqnarray}
The last summation in (\ref{eq3.376}) uses the estimate 
\begin{equation} \label{eq3.392} 
\sum_{1 \leq t\leq p-1}\frac{1}{t}\ll \log p,
\end{equation} 
and $x/p\leq 1$.
\end{proof}
This appears to be the best possible upper bound. The above proof generalizes the sum of resolvents method used in \cite{ML72}. Here, it is reformulated as a finite Fourier transform method, which is applicable to a wide range of functions. A similar upper bound for composite moduli $p=m$ is also proved, [op. cit., equation (2.29)]. 

\begin{thm}  \label{thm3.4}\hypertarget{thm3.4}  Let \(p\geq 2\) be a large prime, and let $\tau $ be a primitive root modulo $p$. Then,
\begin{equation}
	 \sum_{\substack{1\leq n\leq p-1\\\gcd(n,p-1)=1}} e^{i2\pi b \tau^n/p} \ll  p^{1-\varepsilon} \nonumber
\end{equation} 
for any $b \in [1, p-1]$, and any arbitrary small number $\varepsilon \in 
(0, 1/2)$. 
\end{thm}

\begin{proof}[\textbf{Proof}] Let $f(n)=e^{i 2 \pi b\tau^{n} /p}$ and consider the representation
\begin{equation} \label{eq3.392b}
 \sum_{\substack{1\leq n\leq p-1\\\gcd(n,p-1)=1}}  e^{\frac{i2\pi b \tau^n}{p}}= \sum_{ \gcd(n,p-1)=1}\frac{1}{p} \sum_{0 \leq t\leq p-1,}  \sum_{1 \leq s\leq p-1} \omega^{t(n-s)}e^{\frac{i2\pi b \tau^s}{p}} ,
\end{equation}
see \hyperlink{dfn3.23}{Definition} \ref{dfn3.23}. Use the inclusion exclusion principle to rewrite the exponential sum as
\begin{equation}\label{eq3.346}
 \sum_{\substack{1\leq n\leq p-1\\\gcd(n,p-1)=1}} e^{ \frac{i2\pi b \tau^n}{p}} 
= \sum_{ n \leq p-1}\frac{1}{p} \sum_{0 \leq t\leq p-1,}  \sum_{1 \leq s\leq p-1} \omega^{t(n-s)}e^{\frac{i2\pi b \tau^s}{p}} \sum_{\substack{d \mid p-1 \\ d \mid n}}\mu(d)   .
\end{equation} 
The term $t=0$ contributes $-\varphi(p)/p$, and rearranging it yield
\begin{eqnarray}\label{eq3.348}
&& \sum_{\substack{1\leq n\leq p-1\\\gcd(n,p-1)=1}} e^{ \frac{i2\pi b \tau^n}{p}} \\
&=& \sum_{ n \leq p-1}\frac{1}{p} \sum_{1 \leq t\leq p-1,}  \sum_{1 \leq s\leq p-1} \omega^{t(n-s)}e^{\frac{i2\pi b \tau^s}{p}} \sum_{\substack{d \mid p-1 \\ d \mid n}}\mu(d) -\frac{\varphi(p)}{p} \nonumber \\[.3cm]
&=&\frac{1}{p} \sum_{1 \leq t\leq p-1} \left ( \sum_{1 \leq s\leq p-1} \omega^{-ts}e^{\frac{i2\pi b \tau^s}{p}}\right )\left (\sum_{d \mid p-1} \mu(d) \sum_{\substack{n \leq p-1, \\ d \mid n}}   \omega^{tn} \right ) -\frac{\varphi(p)}{p} \nonumber.
\end{eqnarray} 
Taking absolute value, and applying \hyperlink{lem333.24}{Lemma} \ref{lem333.24}, and \hyperlink{lem333.27}{Lemma} \ref{lem333.27}, yield
\begin{eqnarray} \label{eq3.379}
&& \left | \sum_{\substack{1\leq n\leq p-1\\\gcd(n,p-1)=1}} e^{\frac{i2\pi b \tau^n}{p}} \right | \\[.3cm]
&\leq&\frac{1}{p}  \sum_{1 \leq t\leq p-1} \left | \sum_{1 \leq s\leq p-1} \omega^{-ts}e^{i2\pi b \tau^{s}/p} \right | \cdot  \left |\sum_{d \mid p-1} \mu(d) \sum_{\substack{n \leq p-1, \\ d \mid n}}   \omega^{tn} \right | +\frac{\varphi(p)}{p}\nonumber \\[.3cm]
&\ll&\frac{1}{p}  \sum_{1 \leq t\leq p-1} \left ( 2p^{1/2} \log p \right ) \cdot  \left ( \frac{4p^{1+\delta}\log \log p}{\pi t} \right )+\frac{\varphi(p)}{p}\nonumber\\[.4cm]
&\ll& p^{1/2+\delta} (\log p)^2 \nonumber.
\end{eqnarray}
The last summation in (\ref{eq3.379}) uses the estimate 
\begin{equation} \label{eq3.392f} \sum_{1 \leq t\leq p-1}\frac{1}{t}\ll \log p\ll \log p
\end{equation} since $\varphi(p)/p \leq 1$. This is restated in the simpler notation $p^{1/2+\delta}(\log p)^2  \leq p^{1-\varepsilon}$ for any arbitrary small number $\varepsilon \in (0,1/2)$. 
\end{proof}

The upper bound given in \hyperlink{thm3.4}{Theorem} \ref{thm3.4} seems to be optimum. A different proof, which has a weaker upper bound, appears in {\color{red}\cite[Theorem 6]{FS00}}, and related results are given in \cite{CC2009}, \cite{FS01}, \cite{GZ05}, and {\color{red}\cite[Theorem 1]{GK05}}.

\subsection{Equivalent Exponential Sums} 
For any fixed $ 0 \ne b \in \mathbb{F}_p$, the map $ \tau^n \longrightarrow b \tau^n$ is one-to-one in $\mathbb{F}_p$. Consequently, the subsets 
\begin{equation} \label{eq3.220}
 \{ \tau^n: \gcd(n,p-1)=1 \}\quad \text { and } \quad  \{ b\tau^n: \gcd(n,p-1)=1 \} \subset \mathbb{F}_p
\end{equation} have the same cardinalities. As a direct consequence the exponential sums 
\begin{equation} \label{3.330}
\sum_{\gcd(n,p-1)=1} e^{i2\pi b \tau^n/p} \quad \text{ and } \quad \sum_{\gcd(n,p-1)=1} e^{i2\pi \tau^n/p},
\end{equation}
have the same upper bound up to an error term. An asymptotic relation for the exponential sums (\ref{3.330}) is provided in \hyperlink{lem333.22}{Lemma} \ref{lem333.22}. This result expresses the first exponential sum in (\ref{3.330}) as a sum of simpler exponential sum and an error term. 
\begin{lem}   \label{lem333.22}\hypertarget{lem333.22}  Let \(p\geq 2\) be a large primes. If $\tau $ be a primitive root modulo $p$, then,
\begin{equation}  \sum_{\substack{1\leq n\leq p-1\\\gcd(n,p-1)=1}} e^{i2\pi b \tau^n/p} =  \sum_{\substack{1\leq n\leq p-1\\\gcd(n,p-1)=1}} e^{i2\pi  \tau^n/p} + O\left (p^{1/2+\delta} (\log p)^2\right),\nonumber
\end{equation} 
where $\delta>0$ and the implied constant depends only on $p$ and uniformly for $ b \in [1, p-1]$. 	
\end{lem} 
\begin{proof}[\textbf{Proof}] For $b\ne 1$, the exponential sum has the representation 
\begin{eqnarray} \label{eq3.320}
&&  \sum_{\substack{1\leq n\leq p-1\\\gcd(n,p-1)=1}} e^{\frac{i2\pi b \tau^n}{p}} \\
&=&\frac{1}{p} \sum_{1 \leq t\leq p-1} \left ( \sum_{1 \leq s\leq p-1} \omega^{-ts}e^{\frac{i2\pi b \tau^s}{p}}\right )\left (\sum_{d \mid p-1} \mu(d) \sum_{\substack{n \leq p-1, \\ d \mid n}}   \omega^{tn} \right ) -\frac{\varphi(p)}{q}\nonumber,
\end{eqnarray} 
confer equation (\ref{eq3.348}) for details. Moreover, for $b=1$, 
\begin{eqnarray} \label{eq3.321}
&&  \sum_{\substack{1\leq n\leq p-1\\\gcd(n,p-1)=1}} e^{\frac{i2\pi  \tau^n}{p}} \\
&=& \frac{1}{p} \sum_{1 \leq t\leq p-1} \left ( \sum_{1 \leq s\leq p-1} \omega^{-ts}e^{\frac{i2\pi  \tau^s}{p}}\right )\left (\sum_{d \mid p-1} \mu(d) \sum_{\substack{n \leq p-1, \\ d \mid n}}   \omega^{tn} \right ) -\frac{\varphi(p)}{p}\nonumber,
\end{eqnarray}
respectively, see (\ref{eq3.348}). Differencing (\ref{eq3.320}) and (\ref{eq3.321}) produces 
\begin{eqnarray} \label{eq3.90}
& &	 \sum_{\substack{1\leq n\leq p-1\\\gcd(n,p-1)=1}} e^{\frac{i2\pi b \tau^n}{p}} - \sum_{\substack{1\leq n\leq p-1\\\gcd(n,p-1)=1}} e^{\frac{i2\pi  \tau^n}{p}} \\[.3cm]
&=&     \frac{1}{p} \sum_{0 \leq t\leq p-1} \left ( \sum_{1 \leq s\leq p-1} \omega^{-ts}e^{\frac{i2\pi b \tau^s}{p}}-\sum_{1 \leq s\leq p-1} \omega^{-ts}e^{\frac{i2\pi  \tau^s}{p}}\right ) \nonumber \\[.3cm]
&&\hskip 2.5 in   \times\left (\sum_{d \mid p-1} \mu(d) \sum_{\substack{n \leq p-1, \\ d \mid n}}   \omega^{tn} \right ) \nonumber.
\end{eqnarray}
By  Lemma \ref{lem333.24}, the relatively prime summation kernel is bounded by
\begin{eqnarray} \label{eq3.93}
 \left |\sum_{d \mid p-1} \mu(d) \sum_{\substack{n \leq p-1, \\ d \mid n}}   \omega^{tn} \right | 
= \left | \sum_{\substack{1\leq n\leq p-1\\\gcd(n,p-1)=1}}\omega^{tn} \right | 
\ll   \frac{4 p^{1+\delta}} {\pi t}, 
\end{eqnarray}
where the implied constant depends on $p$. Next, by \hyperlink{lem333.27}{Lemma} \ref{lem333.27}, the difference of two Mordell sums is bounded by
\begin{eqnarray} \label{eq3.95}
 \left | \sum_{1 \leq s\leq p-1} \omega^{-ts}e^{\frac{i2\pi b \tau^s}{p}}-\sum_{1 \leq s\leq p-1} \omega^{-ts}e^{\frac{i2\pi  \tau^s}{p}}\right  |  &\leq &  6p^{1/2} \log p. 
\end{eqnarray}
Taking absolute value in (\ref{eq3.90}) and replacing (\ref{eq3.93}), and (\ref{eq3.95}), return
\begin{eqnarray} \label{388}
&& \left|   \sum_{\substack{1\leq n\leq p-1\\\gcd(n,p-1)=1}} e^{i2\pi b \tau^n/p} -\sum_{\gcd(n,p-1)=1} e^{i2\pi  \tau^n/p} \right|  \\[.3cm]
& \ll &      \frac{1}{p} \sum_{0 \leq t\leq p-1} \left ( 4p^{1/2} \log p \right ) \cdot \left ( \frac{6 p^{1+\delta}} {t} \right ) \nonumber\\[.3cm]
&\ll & p^{1/2+\delta} (\log p)^2  \nonumber,
\end{eqnarray}
where $\delta>0$ is a small real number and the implied constant depends on $p$.
\end{proof}
The same proof works for many other subsets of elements $\mathcal{A} \subset \mathbb{F}_p$. For example, 
\begin{equation}
\sum_{n \in \mathcal{A}} e^{i2\pi b \tau^n/p} = \sum_{n \in \mathcal{A}} e^{i2\pi  \tau^n/p} + O(p^{1/2+\delta} \log^c p), 
\end{equation} 
for some constant $c>0$. 

\subsection{Finite Summation Kernels}
\begin{lem}   \label{lem333.20}\hypertarget{lem333.20}  Let \(p\geq 2\) and $p$ be a large prime. Let $\omega=e^{i2 \pi/p} $ be a $p$th root of unity, and let $t \in [1, p-1]$. Then,
\begin{enumerate}[font=\normalfont, label=(\roman*)]
\item $\displaystyle\sum_{n \leq p-1} \omega^{tn} = \frac{\omega^{t}-\omega^{tp}}{1-\omega^{t}},$ 
    
 \item $\displaystyle \left |  \sum_{n \leq p-1} \omega^{tn}  \right |\leq \frac{2p }{\pi t}.$
 \end{enumerate}

\end{lem} 

\begin{proof}[\textbf{Proof}] (i) Use the geometric series to compute this simple exponential sum as
	\begin{eqnarray} \label{eq3.340}
	 \sum_{n \leq p-1} \omega^{tn}
	&=& \frac{\omega^{t}-\omega^{tp}}{1-\omega^{t}} \nonumber.
	\end{eqnarray} 
(ii) Observe that the parameters $p$ is prime, $\omega=e^{i2 \pi/p}$, the integers $t \in [1, p-1]$, and $d \leq p-1$. This data implies that $\pi t/p\ne k \pi $ with $k \in \mathbb{Z}$, so the sine function $\sin(\pi t/p)\ne 0$ is well defined. Using standard manipulations, and $z/2 \leq \sin(z) <z$ for $0<|z|<\pi/2$, the last expression becomes
	\begin{equation}
	\left |\frac{\omega^{t}-\omega^{tp}}{1-\omega^{t}} \right |\leq 	\left | \frac{2}{\sin( \pi t/ p)} \right | 
	\leq \frac{2p}{\pi t}.
	\end{equation}
\end{proof}

\begin{lem}   \label{lem333.24}\hypertarget{lem333.24}  Let \(p\geq 2\) be large a prime, and let $\omega=e^{i2 \pi/p} $ be a $p$th root of unity. Then,
\begin{enumerate}[font=\normalfont, label=(\roman*)]
\item $\displaystyle \sum_{\substack{1\leq n\leq p-1\\\gcd(n,p-1)=1}} \omega^{tn} = \sum_{d \mid p-1} \mu(d) \frac{\omega^{dt}-\omega^{dt((p-1)/d+1)}}{1-\omega^{dt}},$ 
    
 \item $\displaystyle \left |  \sum_{\substack{1\leq n\leq p-1\\\gcd(n,p-1)=1}} \omega^{tn}  \right |\ll \frac{4p^{1+\delta}}{\pi t},$
 \end{enumerate}
where $\delta>0$ is a small number, $\mu(k)$ is the Mobius function, for any fixed pair $d \mid p-1$ and $t \in [1, p-1]$. 
\end{lem} 

\begin{proof}[\textbf{Proof}] (i) Use the inclusion exclusion principle to rewrite the exponential sum as
	\begin{eqnarray} \label{eq360-b}
 \sum_{\substack{1\leq n\leq p-1\\\gcd(n,p-1)=1}} \omega^{tn}&=& \sum_{n \leq p-1} \omega^{tn}  \sum_{\substack{d \mid p-1 \\ d \mid n}}\mu(d)  \\[.3cm]
	&=& \sum_{d \mid p-1} \mu(d) \sum_{\substack{n \leq p-1 \\ d \mid n}} \omega^{tn}\nonumber \\[.3cm]
	& =&\sum_{d\mid p-1} \mu(d) \sum_{m \leq (p-1)/ d} \omega^{dtm} \nonumber\\[.3cm]
	&=& \sum_{d \mid p-1} \mu(d) \frac{\omega^{dt}-\omega^{dt((p-1)/d+1)}}{1-\omega^{dt}} \nonumber.
	\end{eqnarray} 
(ii) Observe that the parameters $q=p$ is prime, $\omega=e^{i2 \pi/p}$, the integers $t \in [1, p-1]$, and $d \leq p-1$. This data implies that $\pi dt/p\ne k \pi $ with $k \in \mathbb{Z}$, so the sine function $\sin(\pi dt/p)\ne 0$ is well defined. Next, observe that \eqref{eq360-b} implies that 
\begin{equation} \label{eq360-c}
	\left | \frac{2}{\sin( \pi dt/ p)} \right | 
	\leq \min \left \{ \frac{p-1}{d},\frac{p}{\pi dt}\right \}.
\end{equation}
Now, since $t\geq 1$ is an integer, any partition of the divisors such as $d\leq D=p/2t$ or $d> D=p/2t$ yields
		\begin{equation}\label{eq360-f}
		\left |\frac{\omega^{dt}-\omega^{dtp}}{1-\omega^{dt}} \right |\leq 	\left | \frac{2}{\sin( \pi dt/ p)} \right | 
		\leq \frac{p}{\pi dt}
	\end{equation} 
Thus, the upper bound is
	\begin{eqnarray}
	\left|   \sum_{d \mid p-1} \mu(d) \frac{\omega^{dt}-\omega^{dt((p-1)/d+1)}}{1-\omega^{dt}} \right| 
	&\leq& 2 \sum_{d \mid p-1}\left( \frac{p}{\pi dt}+ \frac{p}{\pi dt}\right)  \\[.3cm]
	&\leq&\frac{4p}{\pi t} \sum_{d \mid p-1} 1 \nonumber  \\[.3cm]
	&\ll& \frac{4p^{1+\delta}}{\pi t} \nonumber,
	\end{eqnarray}
where the implied constant depends on $p$ and the last inequality uses the elementary estimate $ \sum_{d \mid n} 1 \ll p^{\delta}$.
\end{proof}

\section{Upper Bound for the Error Terms}  \label{s899}
The upper bounds for exponential sums over subsets of elements in finite fields $\mathbb{F}_p$ studied in \hyperlink{S4400}{Section} \ref{S4400} are used to estimate the error terms $E(x,y)$ and $E(x,\Lambda)$ in the proofs of \hyperlink{thm1.1}{Theorem} \ref{thm1.1} and \hyperlink{thm1.2}{Theorem} \ref{thm1.2} respectively. 

\subsection{Error Term for Primitive Roots in Short Intervals}
\begin{lem} \label{lem899.06}\hypertarget{lem899.06}  Let \(p\geq 2\) be a large prime, let $N\gg (\log p)(\log\log p)^2$, let \(\psi(z)=e^{i2\pi z/p}\ne 1\) be an additive character, and let \(\tau\) be a primitive root mod \(p\). If $x\geq1$ and the element $u\in [x,x
	+y]=[M,M+N]$ is not a primitive root, then, 
	\begin{equation} \label{el899.00d}
	\frac{1}{p}\sum_{x \leq u\leq x+ y,}
	 \sum_{\substack{1\leq n\leq p-1\\\gcd(n,p-1)=1}} \sum_{ 0<m \leq p-1} \psi \left((\tau ^n-u)m\right)\ll \frac{y }{p^{\varepsilon}} \nonumber
	\end{equation} 
	for all sufficiently large numbers $1 \leq x< x+y\leq p$, and an arbitrarily small number \(\varepsilon >0\).
\end{lem}

\begin{proof}[\textbf{Proof}]  By hypothesis $\tau ^n-u\ne 0$ for  $u\in [x,x
	+y]$, so $\sum_{ 0<m\leq p-1} \psi \left((\tau ^n-u)m\right)= -1$. This implies the trivial upper bound  
\begin{equation}
\left | E(x,y) \right | < \left |-\frac{\varphi(p-1)}{p}(x+y-x)\right | \leq \frac{x+y-x}{2},
\end{equation}
where $ \varphi(p-1)/p\leq 1/2$. To compute a nontrivial error term, rearrange the triple finite sum in the form
\begin{eqnarray} \label{eq899.05}
E(x,y)&=&\frac{1}{p}\sum_{x \leq u \leq x+ y,} \sum_{ 0<m\leq p-1,}   \sum_{\substack{1\leq n\leq p-1\\\gcd(n,p-1)=1}} \psi ((\tau ^n-u)m) \\  
&= & \frac{1}{p}\sum_{x \leq u \leq x+ y} \left (\sum_{ 0<m\leq p-1,} e^{\frac{-i 2 \pi um}{p}} \right ) \left ( \sum_{\substack{1\leq n\leq p-1\\\gcd(n,p-1)=1}} e^{\frac{i 2 \pi m\tau ^n}{p}} \right )\nonumber \\[.3cm]
 &= & \frac{1}{p}\sum_{x \leq u \leq x+y} \left (\sum_{ 0<m\leq p-1,} e^{\frac{-i 2 \pi um}{p}} \right ) \left (  \sum_{\substack{1\leq n\leq p-1\\\gcd(n,p-1)=1}} e^{\frac{i2\pi  \tau^{n}}{p}} + O(p^{1/2+\delta} (\log p)^2) \right )\nonumber \\[.3cm]
 &= & \frac{1}{p}\sum_{x \leq u \leq x+y} U_p \cdot V_p \nonumber,
\end{eqnarray} 
where $\delta>0$ and the implied constant depends only on $p$, this follows from \hyperlink{lem333.22}{Lemma} \ref{lem333.22}. The first exponential sum $U_p$ has the exact evaluation
\begin{equation}\label{eq899.13}
| U_p| = \left |\sum_{ 0<m\leq p-1} e^{\frac{-i 2 \pi um}{p}} \right |=1,
\end{equation} 
where $\sum_{ 0<m\leq p-1} e^{i 2 \pi um/p}=-1$ for any $u \in [x,x+y]$. The second exponential sum $V_p$ has the upper bound
\begin{eqnarray} \label{eq899.15}
|V_p|&=& \left | \sum_{\substack{1\leq n\leq p-1\\\gcd(n,p-1)=1}} e^{\frac{i2 \pi \tau ^n}{p}}+ O\left (p^{1/2+\delta} (\log p)^2  \right ) \right | \\[.3cm]
&\ll &\left | \sum_{\substack{1\leq n\leq p-1\\\gcd(n,p-1)=1}} e^{\frac{i2 \pi \tau ^n}{p}} \right |+p^{1/2+\delta} (\log p)^2 \nonumber \\[.3cm]
&\ll&  p^{1-\varepsilon} \nonumber,
\end{eqnarray} 
where $\delta>0$ and $\varepsilon\in (0,1/2) $ depends only on $p$, see \hyperlink{thm3.4}{Theorem} \ref{thm3.4}. \\

Taking absolute value in (\ref{eq899.05}), and replacing the estimates (\ref{eq899.13}) and (\ref{eq899.15}) return
\begin{eqnarray} \label{el89991}
\frac{1}{p}\sum_{x \leq u \leq x+y}
\left | U_p \cdot V_p \right | 
&\leq & 
\frac{1}{p} \sum_{x \leq u \leq x+ y} \left | U_p \right | \cdot  |V_p|  \\[.3cm]
&\ll &\frac{1}{p} \sum_{x \leq u \leq x+y} 1 \cdot    p^{1-\varepsilon }\nonumber \\[.3cm]
&\ll &  \frac{ 1}{p^{\varepsilon }}\sum_{x \leq u \leq x+ y}  1 \nonumber \\[.3cm]
&\ll & \frac{y} {p^{\varepsilon}}\nonumber,
\end{eqnarray}
These complete the verification.
\end{proof}

\subsection{Error Term for Primitive Roots in Long Intervals}
The results available in the literature for primes in small intervals of the forms $[M,M+N]=[x, x+y]$ with $y=N < x^{1/2}$ are not uniform. In light of this fact, only the error term for 
the simpler intervals $[2,x]$ can be computed effectively.  
\begin{lem} \label{lem899.16}\hypertarget{lem899.16}  Let \(p\geq 2\) be a large prime, let $x\gg p^{0.535}$, let \(\psi(z)=e^{i2\pi z/p}\ne 1\) be an additive character, and let \(\tau\) be a primitive root mod \(p\). If the element $u\in[2,x]$ is not a primitive root, then,  
	\begin{equation} \label{el899.00}
	\frac{1}{p}\sum_{ u\leq x}
	 \sum_{\substack{1\leq n\leq p-1\\\gcd(n,p-1)=1}} \sum_{ 0<m\leq p-1} \psi \left((\tau ^n-u)m\right) \Lambda(u) \ll \frac{x }{p^{\varepsilon}} \nonumber
	\end{equation} 
	for all sufficiently large numbers $x \geq 1$, and an arbitrarily small number \(\varepsilon >0\).
\end{lem}

\begin{proof}[\textbf{Proof}] Same as the previous one.
\end{proof}

\section{Asymptotics for the Main Terms} \label{S999}\hypertarget{S999}
The notation $f(x)\asymp g(x)$ is defined by $af(x)<g(x)<bf(x)$ for some constants $a,b >0$. 

\subsection{Main Term for Primitive Roots in Short Intervals}
The simpler notation $[M,M+N]=[x,x+y]$ is used in the proof below. 
\begin{lem} \label{lem999.76}\hypertarget{lem999.76}  Let \(p\geq 2\) be a large prime, and let $1 \leq x <x+y<p$ be a pair of numbers. Then, 
\begin{equation} \label{el999.38}
 \sum _{x \leq u\leq x+ y} \frac{1}{p} \sum_{\substack{1\leq n\leq p-1\\\gcd(n,p-1)=1}}1\gg \frac{y}{\log \log p}\left (1+O\left((\log \log p) e^{-c_0 \sqrt{ \log \log p}} \right ) \right )  .\nonumber
	\end{equation} 

\end{lem}

\begin{proof}[\textbf{Proof}] The maximal number $\omega(p-1)$ of prime divisors of highly composite totients $p-1$ satisfies $\omega(p-1) \gg \log p/ \log \log p$. This implies that $z \asymp \log p$. An application of \hyperlink{lem533.21}{Lemma} \ref{lem533.21} to the ratio returns
\begin{eqnarray}
\frac{\varphi(p-1)}{p}&=&\frac{p-1}{p} \frac{1}{p-1}\prod_{q \mid p-1}\left( 1- \frac{1}{q} \right)\\[.3cm] 
&\geq& \prod_{q \leq z}\left( 1- \frac{1}{q} \right)\nonumber\\[.3cm] 
&=&\frac{1}{e^{\gamma} \log z}+ O\left (e^{-c_0 \sqrt{ \log z}}\right )\nonumber\\[.3cm] 
&\gg&\frac{1}{e^{\gamma} \log \log p}+ O\left (e^{-c_0 \sqrt{ \log \log p}}\right )\nonumber.
\end{eqnarray}
Substituting this, the main term reduces to

\begin{eqnarray} \label{el999.53}
M(x,y)&=&\sum _{x \leq u\leq x+y} \frac{1}{p} \sum_{\substack{1\leq n\leq p-1\\\gcd(n,p-1)=1}}1 \\[.3cm]
&=& \frac{\varphi(p-1)}{p}\left ( x+y-x \right )\nonumber \\[.3cm]
&\gg& \left ( \frac{1}{e^{\gamma} \log \log p}+ O\left (e^{-c_0 \sqrt{ \log \log p}}\right )\right )\cdot  y   \nonumber .
\end{eqnarray}
The proves the claim.
\end{proof}

\subsection{Main Term for Primitive Roots in Long Intervals}
\begin{lem} \label{lem999.86}\hypertarget{lem999.86}  Let \(p\geq 2\) be a large prime, and let \(x < p \) be a number. Then, 
\begin{equation} \label{el999.48}
 \sum _{u\leq x} \frac{1}{p} \sum_{\substack{1\leq n\leq p-1\\\gcd(n,p-1)=1}}\Lambda(u)\gg \frac{x}{\log \log p}\left (1+ O\left ( \frac{e^{\gamma} \log \log p}{e^{c_0 \sqrt{ \log \log p}}}\right )\right )  \nonumber
	\end{equation} 
for some constant $c_0>0$.
\end{lem}

\begin{proof}[\textbf{Proof}] The maximal number $\omega(p-1)$ of prime divisors of highly composite totients $p-1$ satisfies $\omega(p-1) \gg \log p/ \log \log p$. This implies that $z \asymp \log p$. An application of \hyperlink{lem533.21}{Lemma} \ref{lem533.21} to the ratio returns
\begin{eqnarray}
\frac{\varphi(p-1)}{p}&=&\frac{p-1}{p} \frac{1}{p-1}\prod_{q \mid p-1}\left( 1- \frac{1}{q} \right)\\[.3cm] 
&\geq& \prod_{q \leq z}\left( 1- \frac{1}{q} \right)\nonumber\\[.3cm] 
&=&\frac{1}{e^{\gamma} \log z}+ O\left (e^{-c_0 \sqrt{ \log z}}\right )\nonumber\\[.3cm] 
&\gg&\frac{1}{e^{\gamma} \log \log p}+ O\left (e^{-c_0 \sqrt{ \log \log p}}\right )\nonumber.
\end{eqnarray}
In addition, using the prime number theorem in the form $\sum_{n \leq x}\Lambda(n)=x +O\left (xe^{-c_0 \sqrt{ \log x}}\right )$, the main term reduces to

\begin{eqnarray} \label{el999.66}
M(x,\Lambda)&=&\sum _{u\leq x} \frac{1}{p} \sum_{\substack{1\leq n\leq p-1\\\gcd(n,p-1)=1}}\Lambda(u) \\[.3cm]
&=& \frac{\varphi(p-1)}{p}\sum _{u\leq x}\Lambda(u) \nonumber\\[.3cm]
&=& \frac{\varphi(p-1)}{p}\left ( x+ O\left (xe^{-c_0 \sqrt{ \log x}}\right )\right )\nonumber \\[.3cm]
&\gg& \left ( \frac{1}{e^{\gamma} \log \log p}+ O\left (e^{-c_0 \sqrt{ \log \log p}}\right )\right )\left ( x+ O\left (xe^{-c_0 \sqrt{ \log x}}\right )\right )   \nonumber \\[.3cm]
&\gg& \frac{x}{\log \log p}\left (1+O\left((\log \log p) e^{-c_0 \sqrt{ \log \log p}} \right ) \right ) \left ( 1+ O\left (e^{-c_0 \sqrt{ \log x}}\right )\right ) \nonumber\\[.3cm]
&\gg& \frac{x}{\log \log p}\left (1+ O\left ( \frac{e^{\gamma} \log \log p}{e^{c_0 \sqrt{ \log \log p}}}\right )\right )   \nonumber.
\end{eqnarray}
This proves the claim.
\end{proof}

\subsection{Main Term for Prime Primitive Roots in Short Intervals}
\begin{lem} \label{lem999.96}\hypertarget{lem999.96}  Let \(p\geq 2\) be a large prime, and let $1 \leq p^{.535} <N<p$ be a pair of numbers. Then,for any number $M <p$, 
\begin{equation} \label{el999.39}
 \sum _{M \leq u\leq M+N} \frac{1}{p} \sum_{\substack{1\leq n\leq p-1\\\gcd(n,p-1)=1}}\Lambda(u)\gg \frac{N}{e^{\gamma} \log \log p} \left (1+ O\left ( \frac{e^{\gamma} \log \log p}{e^{c_0 \sqrt{ \log \log p}}}\right )\right ) .\nonumber
	\end{equation} 

\end{lem}

\begin{proof}[\textbf{Proof}] The maximal number $\omega(p-1)$ of prime divisors of highly composite totients $p-1$ satisfies $\omega(p-1) \gg \log p/ \log \log p$. This implies that $z \asymp \log p$. An application of \hyperlink{lem533.21}{Lemma} \ref{lem533.21} to the ratio returns
\begin{eqnarray}
\frac{\varphi(p-1)}{p}&=&\frac{p-1}{p} \frac{1}{p-1}\prod_{q \mid p-1}\left( 1- \frac{1}{q} \right)\\[.3cm] 
&\geq& \prod_{q \leq z}\left( 1- \frac{1}{q} \right)\nonumber\\[.3cm] 
&=&\frac{1}{e^{\gamma} \log z}+ O\left (e^{-c_0 \sqrt{ \log z}}\right )\nonumber\\[.3cm] 
&\gg&\frac{1}{e^{\gamma} \log \log p}+ O\left (e^{-c_0 \sqrt{ \log \log p}}\right )\nonumber.
\end{eqnarray}
Let $x=M$, and $y=M+N$. Substituting this, the main term reduces to

\begin{eqnarray} \label{el999.51d}
M(x,y,\Lambda)&=&\sum _{x \leq u\leq y} \frac{1}{p} \sum_{\substack{1\leq n\leq p-1\\\gcd(n,p-1)=1}}\Lambda(u)\\[.3cm]
&=& \frac{\varphi(p-1)}{p}\sum _{x \leq u\leq y} \Lambda(u)\nonumber\\[.3cm]
&\gg& \left ( \frac{1}{e^{\gamma} \log \log p}+ O\left (e^{-c_0 \sqrt{ \log \log p}}\right )\right )\sum _{x \leq u\leq y} \Lambda(u)   \nonumber  .
\end{eqnarray}
Applying the prime number theorem in short intervals $\sum _{x \leq n\leq y} \Lambda(n) \gg y-x=N$, see \cite{BP01}, to the last inequality yields
\begin{eqnarray} \label{el999.51g}
M(x,y,\Lambda)&\gg& \left ( \frac{1}{e^{\gamma} \log \log p}+ O\left (e^{-c_0 \sqrt{ \log \log p}}\right )\right )\left ( y-x\right )   \\[.3cm]
&\gg& \frac{N}{e^{\gamma} \log \log p} \left (1+ O\left ( \frac{e^{\gamma} \log \log p}{e^{c_0 \sqrt{ \log \log p}}}\right )\right ) \nonumber .
\end{eqnarray}
The proves the claim.
\end{proof}


\section{Primitive Roots in Short Intervals --- Theorem 1.1} \label{S887}\hypertarget{S887}
The previous sections provide sufficient background materials to assemble the proof of the existence of primitive roots in a short interval $\left [M, M+N \right ]$ for any sufficiently 
large prime $p \geq 2$, a number $N \gg (\log p)^{1+\varepsilon} $, and 
the fixed parameters $ M \geq 2$ and $\varepsilon >0$. \\

The analysis below indicates that the local minima of the ratio $\varphi(p-1)/p$ at the highly composite totients $p-1$ are the primary factor determining the size of the short interval.

\begin{proof}[\textbf{Proof}] (\hyperlink{thm1.1}{Theorem} \ref{thm1.1}) Suppose that the short interval $\left [M, M+N \right ]=[x,x+y]$, with $1 \leq x <y<p$, does not contain a primitive root modulo a large primes \(p\geq 2\), and consider the sum of the characteristic function over the short interval, that is, 
\begin{equation} \label{el887.40}
0=\sum _{x \leq u\leq x+ y} \Psi (u).
\end{equation}
Replacing the characteristic function, \hyperlink{lem333.3}{Lemma} \ref{lem333.3}, and expanding the nonexistence equation (\ref{el887.40}) yield
\begin{eqnarray} \label{el887.50}
0&=&\sum _{x \leq u\leq x+y} \Psi (u)  \\[.3cm]
&=&\sum _{x \leq u\leq x+ y}  \left (\frac{1}{p}\sum_{\substack{1\leq n\leq p-1\\\gcd(n,p-1)=1}} \sum_{ 0\leq m\leq p-1} \psi \left((\tau ^n-u)m\right) \right )\nonumber \\[.3cm]
&=& \frac{c_p}{p} \sum _{x \leq u\leq x+y,}  \sum_{\substack{1\leq n\leq p-1\\\gcd(n,p-1)=1}} 1+\frac{1}{p}\sum _{x \leq u\leq x+y,} 
\sum_{\substack{1\leq n\leq p-1\\\gcd(n,p-1)=1}} \sum_{ 0<m\leq p-1} \psi \left((\tau ^n-u)m\right)\nonumber\\[.3cm]
&=&M(x,y) + E(x,y)\nonumber,
\end{eqnarray} 
where $c_p \geq 0$ is a local correction constant depending on the fixed prime $p\geq 2$. The main term $M(x,y)$ is determined by a finite sum over the trivial additive character \(\psi =1\), and the error term $E(x,y)$ is determined by a finite sum over the nontrivial additive characters \(\psi(t) =e^{i 2\pi t  /p}\neq 1\).\\

An application of \hyperlink{lem999.76}{Lemma} \ref{lem999.76} to the main term, and an an application of \hyperlink{lem899.06}{Lemma} \ref{lem899.06} to the error term yield
\begin{eqnarray} \label{el887.60}
\sum _{x \leq u\leq y} \Psi (u)
&=&M(x,y) + E(x,y)  \nonumber\\
&\gg& \left ( \frac{1}{e^{\gamma} \log \log p}+ O\left (e^{-c_0 \sqrt{ \log \log p}}\right )\right ) (x+y-x)+O\left(\frac{y-x}{p^{\varepsilon}} \right ) \nonumber \\[.3cm]
&\gg& \frac{y-x}{\log \log p}\left (1+ O\left ( \frac{e^{\gamma} \log \log p}{e^{c_0 \sqrt{ \log \log p}}}\right )\right )   \nonumber \\[.3cm]
&>&0,
\end{eqnarray} 
where the implied constant $d_p=e^{-\gamma}a_pc_p \geq 0$ depends on local information and the fixed prime $p\geq 2$. However, a short interval $[x,x+y]$ of length $x+y-x=N \gg (\log p)^{1+\varepsilon} > 0$ contradicts the hypothesis  (\ref{el887.40}) 
for all sufficiently large primes $p \geq 2$. Ergo, the short interval $ [M, M+N  ]$ contains a primitive root for any sufficiently large prime $p \geq 2$ and the fixed parameters $ M \geq 2$ and $\varepsilon >0$. 
\end{proof}

\section{Least Prime Primitive Roots --- Theorem 1.2} \label{s888}
A modified version of the previous result demonstrate the existence of prime primitive roots in an interval $\left [2,x \right ]$ for any sufficiently large prime $p \geq 2$. 
The analysis below indicates that the local minima of the ratio $\varphi(p-1)/p$ at the highly composite totients $p-1$, and the number of primes $\sum_{p \leq x}\Lambda(n)$ 
are the primary factors determining the size of the interval $\left [2,x \right ]$.

\begin{proof}[\textbf{Proof}] (\hyperlink{thm1.2}{Theorem} \ref{thm1.2}) Suppose that the interval $[2,x]$, with $1 \leq x <p$, does not contain a prime primitive root modulo a large primes \(p\geq 2\), and consider the sum of the weighted characteristic function over the integers $u \leq x$, that is, 
\begin{equation} \label{el887.80}
0=\sum _{ u\leq x} \Psi (u) \Lambda(u).
\end{equation}\\
Replacing the characteristic function, \hyperlink{lem333.3}{Lemma} \ref{lem333.3}, and expanding the nonexistence equation (\ref{el887.40}) yield
\begin{eqnarray} \label{el887.55}
0&=&\sum _{ u\leq x} \Psi (u) \Lambda(u) \\[.3cm]
&=&\sum _{u\leq x}  \left (\frac{1}{p}\sum_{\substack{1\leq n\leq p-1\\\gcd(n,p-1)=1}} \sum_{ 0\leq m\leq p-1} \psi \left((\tau ^n-u)m\right) \right ) \Lambda(u)\nonumber\\[.3cm]
&=& \frac{c_p}{p} \sum _{ u\leq x} \Lambda(u)\sum_{\substack{1\leq n\leq p-1\\\gcd(n,p-1)=1}} 1+\frac{1}{p}\sum _{ u\leq x} \Lambda(u)
\sum_{\gcd(n,p-1)=1,} \sum_{ 0<m\leq p-1} \psi \left((\tau ^n-u)m\right)\nonumber\\[.3cm]
&=&M(x,\Lambda) + E(x,\Lambda),\nonumber
\end{eqnarray} 
where $c_p \geq 0$ is a local correction constant depending on the fixed prime $p\geq 2$. The main term $M(x,\Lambda)$ is determined by a finite sum over the trivial additive character \(\psi =1\), and the error term $E(x,\Lambda)$ is determined by a finite sum over the nontrivial additive characters \(\psi(t) =e^{i 2\pi t  /p}\neq 1\).\\

An application of \hyperlink{lem999.86}{Lemma} \ref{lem999.86} to the main term, and an application of \hyperlink{lem899.16}{Lemma} \ref{lem899.16} to the error term yield
\begin{eqnarray} \label{el887.65}
\sum _{ u\leq y} \Psi (u)\Lambda(u)
&=&M(x,\Lambda) + E(x,\Lambda) \nonumber\\
&\gg&  \frac{x}{\log \log p}\left (1+O\left((\log \log p) e^{-c_0 \sqrt{ \log \log p}} \right ) \right )  +O\left(\frac{x}{p^{\varepsilon}} \right ) \nonumber \\[.3cm]
&\gg&  \frac{x}{\log \log p}\left (1+ O\left ( \frac{e^{\gamma} \log \log p}{e^{c_0 \sqrt{ \log \log p}}}\right )\right )   \nonumber \\[.3cm]
&>&0 ,
\end{eqnarray} 
where the implied constant $d_p=e^{-\gamma}a_pc_p \geq 0$ depends on local information and the fixed prime $p\geq 2$. But, an interval $[2,x]$ of length $x-2 \gg (\log p)^{1+\varepsilon} > 0$ contradicts the hypothesis  (\ref{el887.80}) 
for all sufficiently large primes $p \geq 2$. Ergo, the short interval $\left [2,x \right ]$ contains a prime primitive root for any sufficiently large prime $p \geq 2$ and a fixed parameter $\varepsilon >0$. 
\end{proof}

\section{Prime Primitive Roots in Short Intervals --- Theorem 1.3} \label{s1088}
The prime number theorem in short intervals $\sum _{M \leq n\leq M+N} \Lambda(n) \gg N$, see \cite{BP01}. A modified version of the previous result will prove the existence of prime primitive roots in short interval $\left [M,M+N \right ]$ for any sufficiently large prime $p \geq 2$, $N \gg p^{.535}$ and any $M<p$. The analysis below indicates that the number of primes $\sum_{M\leq p \leq M+N}\Lambda(n)$ in a short interval $\left [M,M+N \right ]$ is the primary factor determining the size of the interval $N$. The local minima of the ratio $\varphi(p-1)/p$ at the highly composite totients $p-1$ have a minor impact on the analysis.

\begin{proof}[\textbf{Proof}] (\hyperlink{thm1.3}{Theorem} \ref{thm1.3}) Suppose that the interval $[2,x]$, with $1 \leq x <p$, does not contain a prime primitive root modulo a large primes \(p\geq 2\), and consider the sum of the weighted characteristic function over the integers $u \leq x$, that is, 
\begin{equation} \label{el1088.80}
0=\sum _{ M\leq u\leq M+N} \Psi (u) \Lambda(u).
\end{equation}\\
Replacing the characteristic function, \hyperlink{lem333.3}{Lemma} \ref{lem333.3}, and expanding the nonexistence equation (\ref{el887.40}) yield\\
\begin{eqnarray} \label{el1088.55}
0&=&\sum _{ M\leq u\leq M+N} \Psi (u) \Lambda(u) \\[.3cm]
&=&\sum _{ M\leq u\leq M+N}  \left (\frac{1}{p}\sum_{\substack{1\leq n\leq p-1\\\gcd(n,p-1)=1}} \sum_{ 0\leq m\leq p-1} \psi \left((\tau ^n-u)m\right) \right ) \Lambda(u)\nonumber\\[.3cm]
&=&\frac{c_p}{p} \sum _{ M\leq u\leq M+N}\Lambda(u) \sum_{\substack{1\leq n\leq p-1\\\gcd(n,p-1)=1}} 1\nonumber \\[.3cm]
&&\hskip 1 in +\frac{1}{p}\sum _{ M\leq u\leq M+N} \Lambda(u)
\sum_{\substack{1\leq n\leq p-1\\\gcd(n,p-1)=1}} \sum_{ 0<m\leq p-1} \psi \left((\tau ^n-u)m\right)\nonumber\\[.3cm]
&=&M(N,\Lambda) + E(N,\Lambda)\nonumber,
\end{eqnarray} 
where $c_p \geq 0$ is a local correction constant depending on the fixed prime $p\geq 2$. The main term $M(N,\Lambda)$ is determined by a finite sum over the trivial additive character \(\psi =1\), and the error term $E(N,\Lambda)$ is determined by a finite sum over the nontrivial additive characters \(\psi(t) =e^{i 2\pi t  /p}\neq 1\).\\

An application of \hyperlink{lem999.96}{Lemma} \ref{lem999.96} to the main term, and an application of \hyperlink{lem899.16}{Lemma} \ref{lem899.16} to the error term yield
\begin{eqnarray} \label{el1088.65}
\sum _{ M\leq u\leq M+N} \Psi (u)\Lambda(u)
&=&M(N,\Lambda) + E(N,\Lambda) \nonumber\\[.3cm]
&\gg&  \frac{N}{\log \log p}\left (1+O\left((\log \log p) e^{-c_0 \sqrt{ \log \log p}} \right ) \right )  +O\left(\frac{x}{p^{\varepsilon}} \right ) \nonumber \\[.3cm]
&\gg&  \frac{N}{\log \log p}\left (1+ O\left ( \frac{e^{\gamma} \log \log p}{e^{c_0 \sqrt{ \log \log p}}}\right )\right )   \nonumber \\[.3cm]
&>&0 ,
\end{eqnarray} 
where the implied constant $d_p=e^{-\gamma}a_pc_p \geq 0$ depends on local information and the fixed prime $p\geq 2$. But, an interval $[M,M+N]$ of length $N \gg p^{.535} > 0$ contradicts the hypothesis (\ref{el1088.80}) 
for all sufficiently large primes $p \geq 2$. Ergo, the short interval $\left [M,M+N \right ]$ contains a prime primitive root for any sufficiently large prime $p \geq 2$ and a fixed parameter $M\geq 0$. 
\end{proof}

\section{Problems}
\begin{exe} { \normalfont 
Determine an explicit interval $\left [M, M+N \right ]$, where $N \geq c_0(\log \log p)^{1+\varepsilon}$, $c_0>0$ is a constant, and $\varepsilon \leq 2$, such the the interval contains a primitive root for any prime $p\geq p_0$, and $ M \geq 2$.}
\end{exe}

\begin{exe} { \normalfont Let $a_0=\prod_{p >2}\left( 1- 1/p(p-1) \right) = 0.3739558136 \ldots$ be the average probability of a primitive root modulo a prime $p \geq 2$. Determine the length 
$N \geq2$ of the average short interval $\left [M, M+N \right ]$ that contains $N\cdot (0.3739\ldots )^k  (1-0.3739\ldots )^{N-k}\geq k$ primitive roots, where $N \geq (\log \log p)^{1+\varepsilon}\geq k$, $k\geq 1$, and $\varepsilon=1$.}
\end{exe}

\begin{exe} { \normalfont Show that the distribution of primitive root modulo a large Germain prime $p=2^aq+1$ with $q \geq 2$ prime, and $a \geq 1$, has a normal approximation with 
mean $\mu\approx 2^{a-1}q(1-1/q)$ and standard deviation $\sigma \approx \sqrt{2^{a-2}q(1-1/q^2)}$.}
\end{exe}

 \begin{exe} { \normalfont Estimate the number of highly composite totients $p-1$ in a short interval, that is, $$\sum_{\substack{x \leq p \leq x+y\\
\omega(p-1)\gg \log p/\log \log p}}1, $$ where $x \geq 1$ is a large number, and $1 <y<x$.
}
\end{exe}

\end{document}